\documentclass[preprint,12pt]{elsarticle}

\usepackage{amssymb}
\usepackage{bm}
\usepackage{mathtools}
\usepackage{comment}
\usepackage{tikz-cd}
\usepackage{hyperref}% add hypertext capabilities
\hypersetup{pdfborder=0 0 0} % Uncomment to hide the hyperref box.
\usepackage{amsthm}

\newtheorem{proposition}{Proposition}
\newtheorem*{theorem*}{Theorem}
\newtheorem*{lemma*}{Lemma}

\theoremstyle{definition}
\newtheorem{definition}{Definition}

\newtheorem*{question}{Question}

\theoremstyle{remark}

\begin{document}

\begin{frontmatter}



\title{Winding Number Criterion for the Origin to Belong to the Numerical Range of a Matrix on a Loop of Matrices}


\author[inst1]{Cheng Guo}
\ead{guocheng@stanford.edu}
\affiliation[inst1]{organization={Ginzton Laboratory and Department of Electrical Engineering},%Department and Organization
        addressline={Stanford University}, 
            city={Stanford},
            postcode={94305}, 
            state={California},
            country={USA}}

\author[inst1]{Shanhui Fan}
\ead{shanhui@stanford.edu}

\begin{abstract}
Let $A:[0,1]\to GL(n,\mathbb{C})$ be continuous with $A(0)=A(1)$, thus the winding number of $\det A$ is well-defined. If the winding number is not divisible by $n$, then the origin belongs to the numerical range of $A(\phi)$ for some $\phi \in [0,1]$. 
\end{abstract}



\begin{keyword}
Numerical range \sep Winding number \sep Sectorial matrix
\MSC[2020] 15A60 \sep 55M25
\end{keyword}

\end{frontmatter}


\section{Introduction}
\label{sec:introduction}

Let $M\in M_{n}$ be an $n \times n$ complex matrix. The numerical range of $M$ is the subset of $\mathbb{C}$ defined as 
\begin{equation}
W(M) \coloneqq \{x^{*} M x: x \in \mathbb{C}^{n}, x^{*}x = 1\}.    
\end{equation}
One important question is whether the numerical range contains the origin. This question arises naturally in many pure and applied 
mathematics problems. For example, consider the generalized eigenvalue problem 
\begin{equation}
Bx = \lambda Cx,   
\end{equation}
where $B$ and $C$ are $n\times n$ Hermitian matrices. If the pair of Hermitian matrices $(B,C)$ is definite~\cite{crawford1976,stewart1979}, i.e.,  
\begin{equation}
0 \notin W(B+iC),
\end{equation}
then there exists an invertible matrix $X$ such that $X^{*}BX$ and $X^*CX$ are both diagonal~\cite{stewart1979},  which significantly simplifies the study of the generalized eigenvalue problem. 

The set of matrices 
\begin{equation}
    \mathbb{W}_n \coloneqq \{M\in M_n: 0 \notin W(M) \}
\end{equation} 
are called (rotational) sectorial matrices~\cite{ballantine1975,mathias1992,arlinskii2003,furtado2006,drury2014b,zhang2015g,drury2015,raissouli2017,alakhrass2020a,sano2020,tan2020a,alakhrass2021,mao2021a,chen2022a,ringh2022,zhao2022,wang2023}. These matrices exhibit many nice properties. For example, the geometric mean between positive definite matrices can be generalized to sectorial matrices~\cite{drury2014b}. Moreover, sectorial matrices have been used to define the phases of a matrix~\cite{horn1959,wang2020au,zhao2022,wang2023}, which can be used to angularly bound the eigenvalues by majorization-type inequalities~\cite{wang2020au}. It is therefore important to find criteria that detect whether $M \in \mathbb{W}_n$. This problem has been discussed in many previous works~\cite{johnson1990,bourdon2002,psarrakos2002,
knockaert2006,gau2008,diogo2015,wu2019h}. However, there still lacks a simple analytical criterion.

In this paper, we focus on a closely related question for a loop of matrices:
\begin{question}
Let $A:[0,1] \to M_{n}$ be a continuous function with $A(0)=A(1)$, find a simple criterion to detect whether there exists $\phi \in [0,1]$ such that 
\begin{equation}
0 \in W(A(\phi)).  
\end{equation}    
\end{question} 
We are drawn to this question because many problems involve parameterized matrices. It is useful to detect whether the parameterized matrix always belongs to $\mathbb{W}_n$ as the parameter varies~\cite{furtado2004,furtado2006,psarrakos2000}. For example, consider the parameterized generalized eigenvalue problem 
\begin{equation}
B(p)x = \lambda C(p)x,    
\end{equation}
where $B$ and $C$ are $n\times n$ Hermitian matrices that depend continuously on some parameters $p$. It is important to know whether $(B(p),C(p))$ is definite along a path $\gamma:[0,1]\mapsto p=\gamma(t)$ in the parameter space. This problem reduces to the question above when $\gamma$ is a loop. %The cases where $\gamma$ is a loop are of particular interest. 

This paper aims to provide a simple sufficient criterion for this question. Our criterion only needs the dimensionality and determinant of $A(t)$:

We calculate $\det A(t)$ and check whether there exists $\phi \in [0,1]$ such that
\begin{equation}
\det A(\phi) = 0.    
\end{equation}
If true, then $0 \in W(A(\phi))$. If false, then $\det A(t), t\in[0,1]$ maps to a closed path in $\mathbb{C}\backslash\{0\}$. We calculate the winding number of $\det A$ around the origin:
\begin{equation}
\operatorname{wn}(A) \coloneqq \operatorname{wn}(\det A) \in \mathbb{Z}.    
\end{equation}
We check whether $n$ divides $\operatorname{wn}(A)$. We claim that if
\begin{equation} \label{eq:criterion_to_prove}
n \nmid \operatorname{wn}(A),    
\end{equation}
then there must exist $\phi \in[0,1]$ such that $0 \in W(A(\phi))$. 

The criterion (\ref{eq:criterion_to_prove}) is our main result. It is sufficient but not necessary: If 
\begin{equation}
n \mid \operatorname{wn}(A),     
\end{equation}
there may or may not exist $\phi \in[0,1]$ such that $0 \in W(A(\phi))$. For example, consider 
\begin{equation}
A_{k}(t) = e^{i 2\pi kt} A_{0}, \quad \tilde{A}_{k}(t) = e^{i 2\pi kt} \tilde{A}_{0}, \quad k\in\mathbb{Z}, \quad  t\in[0,1].    
\end{equation}
where $A_{0}$ and $\tilde{A}_{0}$ are invertible $n\times n$ matrices with
\begin{equation}
0 \in W(A_{0}), \quad 0 \notin W(\tilde{A}_{0}).
\end{equation}
One can show that 
\begin{equation}
\operatorname{wn}(A_{k}) = \operatorname{wn}(\tilde{A}_{k}) = nk.    
\end{equation}
Hence 
\begin{equation}
n \mid  \operatorname{wn}(A_{k}), \quad n \mid  \operatorname{wn}(\tilde{A}_{k}).  
\end{equation}
Nonetheless,
\begin{equation}
0 \in W(A_{k}(t)), \quad 0 \notin W(\tilde{A}_{k}(t)).    
\end{equation}

The rest of this paper provides detailed proof of our criterion~(\ref{eq:criterion_to_prove}), which is summarized as the main theorem in Sec.~\ref{sec:main_theorem}.

\section{Background}\label{sec:background}

We summarize the necessary background that will be useful in our proof. 

First, we review some properties of the numerical range~\cite{horn2008,gustafson1997,wu2021a}. %See Refs. for more details. The numerical range has many nice properties.
%\begin{notation}
%Let $S \subseteq \mathbb{C}$. $S$ is compact if it is closed and bounded; $S$ is convex if a line segment $L$ joining two points in $S$ satisfies $L \subseteq S$~\cite{munkres2000}.    
%\end{notation}
\begin{proposition}[Toeplitz-Hausdorff~\cite{toeplitz1918,hausdorff1919}]
Let $M\in M_{n}$. Then $W(M)$ is a compact convex subset of $\mathbb{C}$ that contains all the eigenvalues of $M$. Thus,
\begin{equation}\label{eq:conv_subset_W}
\operatorname{conv} \bm{\lambda}(M) \subseteq W(M),  
\end{equation}
where $\operatorname{conv} \bm{\lambda}(M)$ denotes the convex hull of all the eigenvalues of $M$. 
\end{proposition}

Second, we review the functional continuity of matrix eigenvalues~\cite{li2019a}.
\begin{proposition}
[Kato~\cite{kato1995}]
Let $A: [0,1] \to M_{n}$ be a continuous function. Then there exist $n$ continuous functions $\lambda_{1}(t), \dots, \lambda_{n}(t)$ from $[0,1]$ to $\mathbb{C}$ that parameterize the $n$ eigenvalues (counted with algebraic multiplicities) of $A(t)$.
\end{proposition}

Third, we review the concept of the winding number~\cite{hatcher2002,roe2015}. 

\begin{definition}
Let $\gamma: [0,1] \to \mathbb{C}\backslash \{0\}$ be a continuous path with $\gamma(0)=\gamma(1)$. The winding number of $\gamma$ around the origin is 
\begin{equation}
\operatorname{wn}(\gamma) \coloneqq s(1) - s(0) \in \mathbb{Z},    
\end{equation}
where $(\rho,s)$ is the path written in polar coordinates, i.e.,~the lifted path through the covering map
\begin{equation}
    p: \mathbb{R}^+ \times \mathbb{R} \to \mathbb{C}\backslash \{0\}: (\rho_0, s_0) \mapsto \rho_0 e^{i2\pi s_0}.
\end{equation}
\end{definition}

If $\gamma$ is piecewise differentiable, $\operatorname{wn}(\gamma)$ can be calculated via integration:
\begin{equation}
    \operatorname{wn}(\gamma) = \frac{1}{2\pi i}\int_{0}^{1} \frac{1}{\gamma(t)} \frac{\mathrm{d}\gamma(t)}{\mathrm{d}t}  \mathrm{d}t.
\end{equation}

\begin{definition}
Let $A:[0,1]\to GL(n,\mathbb{C})$ be continuous with $A(0)=A(1)$. The winding number of $A$ is defined as the winding number of $\det A$:
\begin{equation}
\operatorname{wn}(A) \coloneqq \operatorname{wn}(\det A) \in \mathbb{Z}. 
\end{equation}
\end{definition}

If $A$ is piecewise differentiable, $\operatorname{wn}(A)$ can be calculated via integration:
\begin{equation}
    \operatorname{wn}(A) = \frac{1}{2\pi i}\int_{0}^{1} \frac{1}{\det A(t)} \frac{\mathrm{d}\det A(t)}{\mathrm{d} t} \mathrm{d}t.
\end{equation}
Or equivalently, using Jacobi's formula~\cite{magnus2019},  
\begin{equation}
\operatorname{wn}(A) = \frac{1}{2\pi i}\int_{0}^{1} \operatorname{Tr}\left[ A^{-1}(t) \frac{\mathrm{d}A(t)}{\mathrm{d} t}\right] \mathrm{d}t.
\end{equation}
Finally, $\operatorname{wn}(A)$ has a simple topological interpretation: It labels the first homotopy class of $A$ in $GL(n,\mathbb{C})$ since~\cite{hall2015}
\begin{equation}
    \pi_1 \left[GL(n,\mathbb{C})\right] \cong \mathbb{Z}. 
\end{equation}

\section{Main Theorem}\label{sec:main_theorem}

Now we state and prove our main theorem:

\begin{theorem*}\label{theorem:COS}
Let $A:[0,1]\to GL(n,\mathbb{C})$ be continuous with $A(0)=A(1)$. If 
\begin{equation}
n \nmid \operatorname{wn}(A),
\end{equation}
then there exists $\phi \in [0,1]$ such that
\begin{equation}\label{eq:theorem:conclusion}
0 \in \operatorname{conv}\bm{\lambda}(A(\phi)) \subseteq W(A(\phi)).    
\end{equation}
\end{theorem*}
\begin{proof}
From Eq.~(\ref{eq:conv_subset_W}), we obtain
\begin{equation}
\operatorname{conv}\bm{\lambda}(A(t)) \subseteq W(A(t)), \quad \forall t \in [0,1].  
\end{equation}
It suffices to prove the existence of $\phi \in [0,1]$ such that 
\begin{equation}\label{eq:relation_to_proof}
0 \in \operatorname{conv}\bm{\lambda}(A(\phi)).   \end{equation}
We denote the $n$ eigenvalues of $A(t)$ as 
\begin{equation}
\bm{\lambda}(A(t)) =  \left[\lambda_{1}(t), \dots, \lambda_{n}(t)\right]^T.  
\end{equation}
According to Kato's theorem, we can choose $\lambda_{1}(t), \dots, \lambda_{n}(t)$ to be continuous functions of $t$ on $[0,1]$. Since $A(t) \in GL(n,\mathbb{C})$, 
\begin{equation}
\lambda_{j}(t) \neq 0, \qquad j=1,\dots,n.    
\end{equation} 
Since $A(0) = A(1)$, 
\begin{equation}
\lambda_{j}(1) = \lambda_{\alpha_{j}}(0),  
\end{equation}
where 
\begin{equation}
\alpha (A) = \begin{pmatrix}
1 & 2 &\dots & n \\
\alpha_{1} & \alpha_{2} &\dots &\alpha_{n}
\end{pmatrix} \in S_n  
\end{equation}
is a permutation. We will prove (\ref{eq:relation_to_proof}) first in the special case when 
\begin{equation}
\alpha(A) = \operatorname{id} \coloneqq \begin{pmatrix}
1 & 2 &\dots & n \\
1 & 2 &\dots & n
\end{pmatrix},    
\end{equation}
then in the other cases when $\alpha(A) \neq \operatorname{id}$. 

\begin{figure}[htbp]
    \centering
    \includegraphics[width=0.4\textwidth]{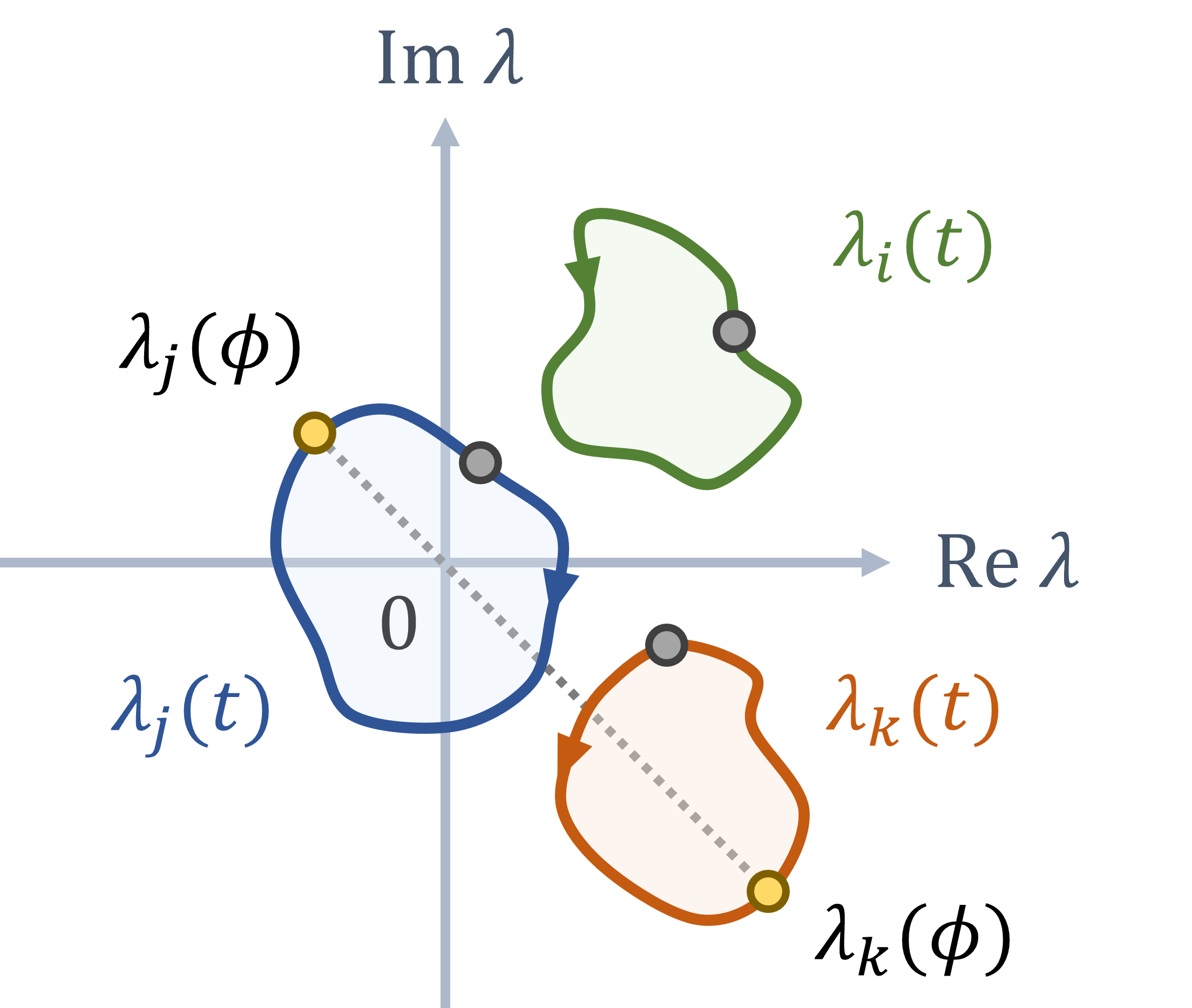}
    \caption{Scheme for the $\alpha(A) = \operatorname{id}$ case.}
    \label{fig:alpha_id_case}
\end{figure}
\textbf{Case I.} $\alpha (A)= \operatorname{id}$. (See Fig.~\ref{fig:alpha_id_case} for a scheme.) Then each $\lambda_j: [0,1] \to \mathbb{C}\backslash \{0\}$ traces out a closed path, and thus has a well-defined winding number:
\begin{equation}\label{eq:def_wind_lambda_j}
\operatorname{wn}(\lambda_{j}) \in \mathbb{Z}, \qquad j=1,\dots,n.    
\end{equation}
The winding number of a pointwise product of loops is the sum of winding numbers of each loop~\cite{roe2015}:  
\begin{equation}
\operatorname{wn}(A) = \operatorname{wn}(\det A) = \operatorname{wn}\left(\prod_{i=1}^{n} \lambda_{j}\right) = \sum_{j=1}^{n} \operatorname{wn}(\lambda_{j}).    
\end{equation}
Since
\begin{equation}\label{eq:condition_not_divisible}
n \nmid \operatorname{wn}(A),   
\end{equation}
there exist $1\leq j\neq k \leq n$ such that
\begin{equation}
\operatorname{wn}(\lambda_{j}) \neq \operatorname{wn}(\lambda_{k}).    
\end{equation}
Otherwise, we would have 
\begin{equation}
n|\operatorname{wn}(A) = n \operatorname{wn}(\lambda_{1}),   
\end{equation}
which contradicts (\ref{eq:condition_not_divisible}). Now consider the line segment $\overline{\lambda_{j}(t)\lambda_{k}(t)}$. We claim that there exists $\phi\in[0,1]$ such that 
\begin{equation}\label{eq:0_line_segment}
0 \in \overline{\lambda_{j}(\phi)\lambda_{k}(\phi)} \subseteq \operatorname{conv}\bm{\lambda}(A(\phi)),  
\end{equation} 
which would complete the proof of  (\ref{eq:relation_to_proof}) for the $\alpha(A) = \operatorname{id}$ case.  

Now we prove (\ref{eq:0_line_segment}). Since $\lambda_{j}(t) \neq 0$ and $\lambda_{k}(t) \neq 0$, it suffices to show the existence of $\phi \in [0,1]$ such that 
\begin{equation}\label{eq:arg_to_prove}
\arg\left[\lambda_{j}(\phi)\right] - \arg\left[\lambda_{k}(\phi)\right] = (2l+1)\pi, \quad l\in \mathbb{Z}.   
\end{equation}
Since $\lambda_{j}(t)$ and $\lambda_{k}(t)$ are continuous, we can choose $\arg[\lambda_{j}(t)]$ and $\arg[\lambda_{k}(t)]$ to be continuous functions of $t$. Then from Eq.~(\ref{eq:def_wind_lambda_j}),
\begin{align}
\arg[\lambda_{j}(1)] -  \arg[\lambda_{j}(0)] &= 2 \pi \cdot \operatorname{wn}(\lambda_{j}),  \\
\arg[\lambda_{k}(1)] -  \arg[\lambda_{k}(0)] &= 2 \pi \cdot \operatorname{wn}(\lambda_{k}).
\end{align}
Now consider the function
\begin{equation}
\Psi(t) \coloneqq \arg[\lambda_{j}(t)] - \arg[\lambda_{k}(t)], \quad t \in [0,1].  
\end{equation}
$\Psi$ is continuous and 
\begin{equation}\label{eq:phase_difference}
\Psi(1) - \Psi(0) = 2\pi \nu, \quad \nu \coloneqq \operatorname{wn}(\lambda_{j})-\operatorname{wn}(\lambda_{k}) \in \mathbb{Z}\backslash\{0\}.    
\end{equation}
We denote 
\begin{equation}
\Psi(0) = 2\pi \chi + \xi, \quad \chi \in \mathbb{Z}, \, \xi \in [-\pi,\pi).    
\end{equation}
Then 
\begin{equation}
\Psi(1) = 2\pi (\nu + \chi) + \xi.    
\end{equation}
We know that the integer $\nu \neq 0$. If $\nu >0$, then 
\begin{equation}
\Psi(0) \leq 2\pi \chi + \pi \leq \Psi(1).    
\end{equation}
The intermediate value theorem implies that there exists $\phi \in [0,1]$ such that 
\begin{equation}
\Psi(\phi) = 2\pi \chi + \pi.    
\end{equation}
We set $l = \chi$. If $\nu <0$, then 
\begin{equation}
\Psi(1) \leq 2\pi \chi - \pi \leq \Psi(0).    
\end{equation}
The intermediate value theorem implies that there exists $\phi \in [0,1]$ such that 
\begin{equation}
\Psi(\phi) = 2\pi \chi - \pi.    
\end{equation}
We set $l = \chi-1$. In any case, we obtain Eq.~(\ref{eq:arg_to_prove}). This completes the proof of (\ref{eq:0_line_segment}). 

We summarized our result so far in its contrapositive form as a lemma:
\begin{lemma*}\label{lemma}
Let $A:[0,1] \to GL(n,\mathbb{C})$ be continuous with $A(0)=A(1)$. If 
\begin{equation}
\alpha(A) = \operatorname{id} \qquad \text{and} \qquad 0 \notin \operatorname{conv}\bm{\lambda}[A(t)], \quad \forall t \in[0,1],    
\end{equation}
then
\begin{equation}
n \mid \operatorname{wn}(A). 
\end{equation}
\end{lemma*}

\begin{figure}[htbp]
    \centering
    \includegraphics[width=1.0\textwidth]{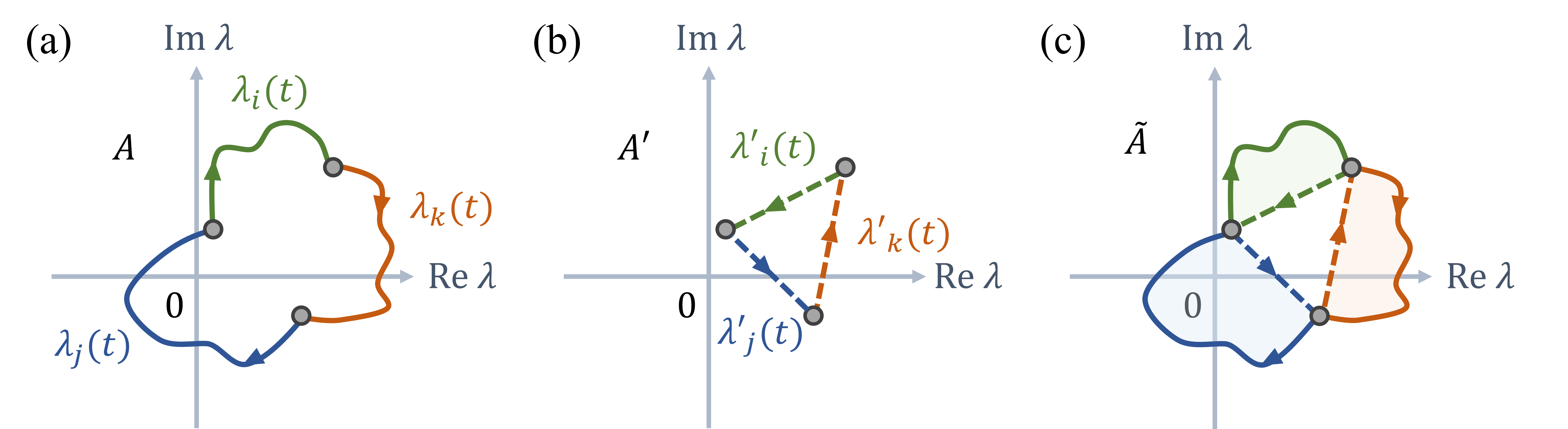}
    \caption{Scheme for the $\alpha(A) \neq \operatorname{id}$ case.}
    \label{fig:alpha_not_id_case}
\end{figure}

\textbf{Case II.} $\alpha(A) \neq \operatorname{id}$. (See Fig.~\ref{fig:alpha_not_id_case} for a scheme.)
We first outline the proof. We prove this case by contradiction. Suppose that there exists $A:[0,1] \to GL(n,\mathbb{C})$ that is continuous with the following properties (Fig.~\ref{fig:alpha_not_id_case}a):
\begin{gather}
A(1)=A(0); \label{eq:Case-II_condition_S_periodic} \\
\lambda_{j}(1) = \lambda_{\alpha_{j}}(0), \quad \alpha(A) \neq \operatorname{id}; \\
\label{eq:Case-II_condition_conv}
0 \notin \operatorname{conv}\bm{\lambda}[A(t)], \quad \forall t \in[0,1]; \\
n\nmid \operatorname{wn}(A).\label{eq:Case-II_condition_not_divide}
\end{gather}
The assumption~(\ref{eq:Case-II_condition_conv}) is opposite to what we want to prove (\ref{eq:theorem:conclusion}). We will show that these assumptions lead to a contradiction with Lemma. We construct a continuous and piecewise differentiable function $A':[0,1] \to M_{n}$ with the following properties (Fig.~\ref{fig:alpha_not_id_case}b):
\begin{gather}
A'(1)=A'(0) = A(1) = A(0); \label{eq:S_prime_property_1}\\
\lambda'_{j}(0) = \lambda_{j}(1), \quad \lambda'_{j}(1) = \lambda_{j}(0);\label{eq:S_prime_property_2}\\
\operatorname{conv}\bm{\lambda}[A'(t)] \subseteq \operatorname{conv}\bm{\lambda}[A(1)], \quad \forall t \in [0,1];\label{eq:S_prime_property_3}\\
\det A'(t) \neq 0, \quad \forall t \in[0,1]; \label{eq:S_prime_property_4}\\
\operatorname{wn}(A') = 0.\label{eq:S_prime_property_5}
\end{gather}
Then we construct the concatenation of $A$ and $A'$ (Fig.~\ref{fig:alpha_not_id_case}c) defined as~\cite{hatcher2002,roe2015}
\begin{equation}
\tilde{A}(t) \coloneqq \begin{cases}
A(2t), &t \in \left[0,\frac{1}{2}\right] \\
A'(2t-1), &t \in \left[\frac{1}{2},1\right]
\end{cases}     
\end{equation}
The winding number of a concatenation of
loops is the sum of the winding numbers of each loop~\cite{roe2015}:
\begin{equation}
\operatorname{wn}(\tilde{A}) = \operatorname{wn}(A) + \operatorname{wn}(A') = \operatorname{wn}(A).    
\end{equation}
Moreover, $\tilde{A}$ satisfies all the conditions of the Lemma, and consequently,
\begin{equation}
n \mid \operatorname{wn}(\tilde{A})=\operatorname{wn}(A)    
\end{equation}
which contradicts our premise (\ref{eq:Case-II_condition_not_divide}). This means that the assumption (\ref{eq:Case-II_condition_conv}) is false, and our original conclusion (\ref{eq:theorem:conclusion}) is true. This completes the proof of Case II. 

Now we fill in the missing details of the proof outlined above.

\vspace{5mm} %5mm vertical space
\textbf{1. Construction of $A'$.} 
\vspace{5mm} %5mm vertical space

We start with the Schur triangulation~\cite{horn2012} of $A(0)=A(1)$:
\begin{equation}
A(0)=A(1) = U\Lambda U^{*} = V \Lambda' V^{*},    
\end{equation}
where $U, V$ are $m\times m$ unitary matrices, and 
\begin{equation}
\Lambda = \begin{pmatrix}
\lambda_{1}(0) & & \bigstar \\
& \ddots & \\
0 & & \lambda_{n}(0)
\end{pmatrix}, \qquad \Lambda' = \begin{pmatrix}
\lambda_{\alpha_{1}}(0) & & \bigstar \\
& \ddots & \\
0 & & \lambda_{\alpha_{n}}(0)
\end{pmatrix}    
\end{equation}
are upper triangular matrices. Here $\bigstar$ denotes possibly nonzero elements. We recall that $\lambda_j(1) = \lambda_{\alpha_j}(0)$, $j=1,\dots, m$. We define two diagonal matrices
\begin{equation}
D = \begin{pmatrix}
\lambda_{1}(0) & & 0 \\
& \ddots & \\
0 & & \lambda_{n}(0)
\end{pmatrix}, \qquad D' = \begin{pmatrix}
\lambda_{\alpha_{1}}(0) & & 0 \\
& \ddots & \\
0 & & \lambda_{\alpha_{n}}(0)
\end{pmatrix}.    
\end{equation}
Now, we introduce a partition of $[0,1]$:
\begin{equation}
0=t_{1}<t_{2}<t_{3}<t_{4}<t_{5} < t_{6} = 1,    
\end{equation} 
and set
\begin{align}
A'(t_{1})&=  V \Lambda' V^{*} = A(1), \label{eq:S_prime_phi1}\\
A'(t_{2})&= \Lambda', \\
A'(t_{3})&= D',  \\
A'(t_{4})&= D,   \\
A'(t_{5})&= \Lambda,  \\
A'(t_{6})&= U\Lambda U^{*} = A(0).\label{eq:S_prime_phi6}
\end{align}
To connect these points in $M_{n}$, we define two skew-Hermitian matrices
\begin{equation}\label{eq:def_J_K}
\quad J = \log U, \qquad K = \log V,    
\end{equation}
then we set
\begin{equation}\label{eq:def_S_prime}
A'(t) = \begin{cases}
e^{K (t_{2}-t)/(t_{2}-t_{1})}\, A'(t_{2})\, e^{-K (t_{2}-t)/(t_{2}-t_{1})}, &t_{1} \leq t < t_{2}\\
\frac{t_{3}-t}{t_{3}-t_{2}}\, A'(t_{2}) +\frac{t-t_{2}}{t_{3}-t_{2}} \,A'(t_{3}), &t_{2} \leq t < t_{3} \\
\frac{t_{4}-t}{t_{4}-t_{3}} \,A'(t_{3}) +\frac{t-t_{3}}{t_{4}-t_{3}} \,A'(t_{4}), &t_{3} \leq t < t_{4} \\
\frac{t_{5}-t}{t_{5}-t_{4}}\, A'(t_{4}) +\frac{t-t_{4}}{t_{5}-t_{4}} \,A'(t_{5}), &t_{4} \leq t < t_{5} \\
e^{J (t-t_{5})/(t_{6}-t_{5})} \, A'(t_{5}) \, e^{-J (t-t_{5})/(t_{6}-t_{5})}, &t_{5} \leq t \leq t_{6}
\end{cases}    
\end{equation}
The evolution of $A'(t)$ is as follows: As $t$ goes from $t_1=0$ to $t_2$, $A'(0) = A(1)$ is continuously deformed into its Schur triangulation $\Lambda'$ by unitary similarity; from $t_2$ to $t_3$,  $\Lambda'$ is continuously reduced to $D'$ by gradually diminishing the off-diagonal elements; from $t_3$ to $t_4$, $D'$ is continuously deformed into $D$ by linear interpolation, which leads to a permutation of the diagonal elements; from $t_4$ to $t_5$, $D$ is continuously restored to another Schur triangulation $\Lambda$ by gradually adding the off-diagonal elements; from $t_5$ to $t_6=1$, $\Lambda$ is continuously deformed into $A'(1) =A(0)$ by unitary similarity. 

\vspace{5mm} %5mm vertical space
\textbf{2. Properties of $A'$.} 
\vspace{5mm} %5mm vertical space

We prove the claimed properties of $A'$. From the definition in Eq.~(\ref{eq:def_S_prime}), we confirm that $A': [0,1]\to M_{n}$ is continuous and piecewise differentiable. 

\begin{itemize}
    \item \emph{Proof of Eq.~(\ref{eq:S_prime_property_1}).} This is a direct result of Eqs.~(\ref{eq:Case-II_condition_S_periodic}), (\ref{eq:S_prime_phi1}) and (\ref{eq:S_prime_phi6}).
\item \emph{Proof of Eq.~(\ref{eq:S_prime_property_2}).} We denote the $n$ eigenvalues of $A'(t)$ as 
\begin{equation}
\bm{\lambda}(A'(t)) = [\lambda'_{1}(t),\dots,\lambda'_{n}(t)]^T.    
\end{equation}
By Kato's theorem, we can choose $\lambda'_{1}(t), \dots, \lambda'_{n}(t)$ to be continuous function of $t$ on $[0,1]$. Since $A'(0)=A(1)$, we can choose the ordering of $\lambda'_{j}$ such that
\begin{equation}
\lambda'_{j}(0) = \lambda_{j}(1) = \lambda_{\alpha_{j}}(0). 
\end{equation}
From the definition of $A'(t)$ [Eq.~(\ref{eq:def_S_prime})], we obtain
\begin{equation}\label{eq:lambda_prime_phi}
\lambda'_{j}(t) = \begin{cases}
\lambda_{\alpha_{j}}(0), & t_{1} \leq t < t_{3} \\
\frac{t_{4}-t}{t_{4}-t_{3}} \lambda_{\alpha_{j}}(0) +\frac{t-t_{3}}{t_{4}-t_{3}} \lambda_{j}(0), & t_{3} \leq t < t_{4}  \\
\lambda_{j}(0), & t_{4} \leq t \leq t_{6}
\end{cases}    
\end{equation}
In particular,
\begin{equation}
\lambda'_{j}(1) = \lambda'_{j}(t_{6}) = \lambda_{j}(0).    
\end{equation}
This completes the proof of Eq.~(\ref{eq:S_prime_property_2}).

\item \emph{Proof of Eq.~(\ref{eq:S_prime_property_3}).}
From Eq.~(\ref{eq:lambda_prime_phi}), we obtain
\begin{align}
&\operatorname{conv}\bm{\lambda}[A'(t)] = \operatorname{conv}\bm{\lambda}[A(1)], &t_{1} \leq t < t_{3};  \label{eq:conv_lambda_phi13}  \\
&\operatorname{conv}\bm{\lambda}[A'(t)] = \operatorname{conv}\bm{\lambda}[A(0)] = \operatorname{conv}\bm{\lambda}[A(1)], &t_{4} \leq t < t_{6}, \label{eq:conv_lambda_phi46}    
\end{align}
where we have used Eq.~(\ref{eq:Case-II_condition_S_periodic}) in Eq.~(\ref{eq:conv_lambda_phi46}). We also note from Eq.~(\ref{eq:lambda_prime_phi}) that
\begin{equation}
\lambda_{j}(t)\in \overline{\lambda_{\alpha_{j}}(0)\lambda_{j}(0)} \subseteq \operatorname{conv}\bm{\lambda}[A(1)],  \quad j=1,\dots,m, \quad  t_{3}\leq t<t_{4}.    \label{eq:inclusion_conv_lambda_phi34}
\end{equation}
By definition,  $\operatorname{conv}\bm{\lambda}[A'(t)]$  is the smallest convex set on $\mathbb{C}$ that contains all $\lambda_{j}(t)$. Thus, Eq.~(\ref{eq:inclusion_conv_lambda_phi34}) implies that
\begin{equation}
\operatorname{conv}\bm{\lambda}[A'(t)] \subseteq \operatorname{conv}\bm{\lambda}[A(1)], \quad  t_{3}\leq t<t_{4}.  \label{eq:conv_lambda_phi34}  
\end{equation}
Combining (\ref{eq:conv_lambda_phi13}), (\ref{eq:conv_lambda_phi46}), and (\ref{eq:conv_lambda_phi34}), we obtain 
\begin{equation}
\operatorname{conv}\bm{\lambda}[A'(t)] \subseteq \operatorname{conv}\bm{\lambda}[A(1)], \quad  \forall t \in [0,1].    
\end{equation}
This completes the proof of~(\ref{eq:S_prime_property_3}).
\item \emph{Proof of (\ref{eq:S_prime_property_4}).} From our assumption~(\ref{eq:Case-II_condition_conv}),
\begin{equation}\label{eq:0_not_in_conv_lambda}
0 \notin \operatorname{conv}\bm{\lambda}[A(1)]. 
\end{equation}
Combining (\ref{eq:0_not_in_conv_lambda}) and (\ref{eq:S_prime_property_3}), we obtain
\begin{equation}
0 \notin \operatorname{conv}\bm{\lambda}[A'(t)], \quad \forall t \in [0,1],    
\end{equation}
which implies that
\begin{equation}
\lambda'_{j}(t) \neq 0, \quad \forall t \in [0,1], \quad j=1,\dots, m.    
\end{equation}
Therefore,
\begin{equation}
\det A'(t) = \prod_{j=1}^{n} \lambda'_{j}(t) \neq 0, \quad \forall t \in [0,1].    
\end{equation}
This completes the proof of Eq.~(\ref{eq:S_prime_property_4}).

\item \emph{Proof of Eq.~(\ref{eq:S_prime_property_5}).} Since $A': [0,1]\to M_{n}$ is continuous and piecewise differentiable with the properties~(\ref{eq:S_prime_property_1}) and (\ref{eq:S_prime_property_4}), it has a well-defined winding number 
\begin{equation}
\operatorname{wn}(A') \coloneqq \frac{1}{2\pi i}\int_0^{1} \frac{1}{\det A'} \frac{\mathrm{d} \det A'}{\mathrm{d} t}  \, \mathrm{d}t  = \sum_{j=1}^{n}\frac{1}{2\pi i}\int_0^{1} \frac{1}{\lambda'_{j}} \frac{\mathrm{d} \lambda'_{j}}{\mathrm{d} t}  \, \mathrm{d}t. 
\end{equation}
Eq.~(\ref{eq:lambda_prime_phi}) shows that 
\begin{equation}
\frac{\mathrm{d} \lambda'_{j}}{\mathrm{d} t} = 0, \qquad t \in [0,t_3] \cup [t_4,1]. 
\end{equation}
Therefore, 
\begin{equation}\label{eq:wind_S_prime_phi34}
\operatorname{wn}(A')  = \sum_{j=1}^{n}\frac{1}{2\pi i}\int_{t_{3}}^{t_{4}} \frac{1}{\lambda'_{j}} \frac{\mathrm{d} \lambda'_{j}}{\mathrm{d} t}  \, \mathrm{d}t. 
\end{equation}
We will show that $\operatorname{wn}(A')  = 0$. Since any permutation can be uniquely expressed as a product of disjoint cycles up to the order of the cycles, we have
\begin{equation}
    \alpha(A) = \beta_1 \beta_2 \cdots \beta_q
\end{equation}
where $\beta_1, \beta_2,\dots, \beta_q$ are disjoint cycles of length $l_1, l_2, \dots, l_q$, respectively, so that  
\begin{equation}
    n = \sum_{k=1}^q l_k.
\end{equation}
(Fixed points are included as $1$-cycles.) These disjoint cycles partition the set $\{1,2, \ldots, n\}$ into disjoint subsets $E_1, E_2, \ldots, E_q$. We can rewrite Eq.~(\ref{eq:wind_S_prime_phi34}) as    
\begin{equation}\label{eq:wind_S_prime_sum_mu_k}
\operatorname{wn}(A')  = \sum_{k=1}^{q} \sum_{j \in E_k} 
\frac{1}{2\pi i}\int_{t_{3}}^{t_{4}} \frac{1}{\lambda'_{j}} \frac{\mathrm{d} \lambda'_{j}}{\mathrm{d} t}  \, \mathrm{d}t = \sum_{k=1}^{q} \mu_k,  
\end{equation}
where we define
\begin{equation}\label{eq:def_mu_k}
\mu_k \coloneqq \sum_{j \in E_k} 
\frac{1}{2\pi i}\int_{t_{3}}^{t_{4}} \frac{1}{\lambda'_{j}} \frac{\mathrm{d} \lambda'_{j}}{\mathrm{d} t}  \, \mathrm{d}t.     
\end{equation}
We claim that 
\begin{equation} \label{eq:claim_mu_k}
    \mu_k = 0, \qquad k=1,\ldots,q. 
\end{equation}
If Eq.~(\ref{eq:claim_mu_k}) is true, then combining Eqs.~(\ref{eq:claim_mu_k}) and (\ref{eq:wind_S_prime_sum_mu_k}), we obtain:
\begin{equation}
\operatorname{wn}(A') = 0.    
\end{equation} 
which completes the proof of Eq.~(\ref{eq:S_prime_property_5}).

Now we prove Eq.~(\ref{eq:claim_mu_k}).

If $l_k = 1$, then $E_k = \{j\}$ with $\alpha_j = j$. From Eq.~(\ref{eq:lambda_prime_phi}),
\begin{equation}
\lambda'_j(t) = \lambda_j(0),  \qquad t_{3} \leq t < t_{4},  
\end{equation}
thus $\mu_k = 0$. 

If $l_k = 2$, then $E_k = \{j, j'\}$ with $\alpha_j = j'$ and $\alpha_{j'} = j$. From Eq.~(\ref{eq:lambda_prime_phi}),
\begin{align}
\lambda'_j(t) &= \frac{t_{4}-t}{t_{4}-t_{3}} \lambda_{j'}(0) +\frac{t-t_{3}}{t_{4}-t_{3}} \lambda_{j}(0),  \qquad t_{3} \leq t < t_{4};  \\
\lambda'_{j'}(t) &= \frac{t_{4}-t}{t_{4}-t_{3}} \lambda_{j}(0) +\frac{t-t_{3}}{t_{4}-t_{3}} \lambda_{j'}(0),  \qquad t_{3} \leq t < t_{4}.
\end{align}
Thus we have 
\begin{equation}
\lambda'_{j'}(t) = \lambda'_{j}(t_3+ t_4 - t);
\end{equation}
\begin{equation}
\frac{\mathrm{d}}{\mathrm{d} t }\lambda'_{j}(t) = -\frac{\mathrm{d}}{\mathrm{d} t }\lambda'_{j'}(t) = \frac{\lambda'_{j}(0)-\lambda'_{j'}(0)}{t_4 -t_3}.
\end{equation}
Therefore, 
\begin{align}
\mu_k &= 
\frac{1}{2\pi i}\int_{t_{3}}^{t_{4}} \left\{\frac{1}{\lambda'_{j}} \frac{\mathrm{d} \lambda'_{j}}{\mathrm{d} t} + \frac{1}{\lambda'_{j'}} \frac{\mathrm{d} \lambda'_{j'}}{\mathrm{d} t} \right\}\, \mathrm{d}t \\
&=\frac{1}{2\pi i} \frac{\lambda'_{j}(0)-\lambda'_{j'}(0)}{t_4 -t_3} \int_{t_{3}}^{t_{4}} \left\{\frac{1}{\lambda'_{j}(t)} - \frac{1}{\lambda'_{j'}(t)} \right\}\, \mathrm{d}t \\
&=\frac{1}{2\pi i} \frac{\lambda'_{j}(0)-\lambda'_{j'}(0)}{t_4 -t_3} \int_{t_{3}}^{t_{4}} \left\{\frac{1}{\lambda'_{j}(t)} - \frac{1}{\lambda'_{j}(t_3 + t_4 - t)} \right\}\, \mathrm{d}t \\
&= 0.
\end{align}
The last equality follows from the identity:
\begin{equation}
\int_a^b f(x) \, \mathrm{d}x = \int_a^b f(a+b -x) \, \mathrm{d}x.   
\end{equation}
If $l_k \geq 3$, then $E_k = \{j_1, j_2, \ldots, j_{l_k}\}$. We arrange them such that 
\begin{equation}\label{eq:j_s_arrangement}
\alpha_{j_s} = j_{s-1}, \quad s=1,2, \dots, l_k,  
\end{equation}
where we define $j_0 \coloneqq j_{l_k}$. Substituting Eq.~(\ref{eq:j_s_arrangement}) into Eq.~(\ref{eq:lambda_prime_phi}), we obtain
\begin{equation}
\lambda'_{j_s}(t) = \frac{t_{4}-t}{t_{4}-t_{3}} \lambda_{j_{s}-1}(0) +\frac{t-t_{3}}{t_{4}-t_{3}} \lambda_{j_s}(0), \quad  t_{3} \leq t < t_{4}.   
\end{equation}
As $t$ runs over $[t_3, t_4]$, $\lambda'_{j_s}(t)$ traces out a directed line segment $\lambda_{j_{s}-1}(0) \rightarrow \lambda_{j_{s}}(0)$. Enumerating $j_s \in E_k$, we obtain a (possibly non-simple) directed polygon $P_k$:
\begin{equation}
P_k \coloneqq \begin{tikzcd}
  \lambda_{j_1}(0) \ar[r,"\lambda'_{j_2}(t)"]
& \lambda_{j_2}(0) \ar[r,"\lambda'_{j_3}(t)"]
& \lambda_{j_{3}}(0) \;\; \cdots \;\; 
   % This arrow isn't in the picture, but I chose to add it. Feel free to remove.
\lambda_{j_{l_k -1}}(0) \ar[r,"\lambda'_{j_{l_k}}(t)"]%\ar[l,"\lambda'_{j_{l_k -1}}(t)"']
& \lambda_{j_{l_k}}(0)  \ar[lll,bend left=50,looseness=0.25,"\lambda'_{j_1}(t)"]
\end{tikzcd}.
\end{equation}
Fig.~\ref{fig:alpha_not_id_case}b shows an instance of a directed polygon with $l_k = 3$. Now we see that $\mu_k$ as defined in Eq.~(\ref{eq:def_mu_k}) equals the winding number of the directed polygon $P_k$ around the origin~\cite{hormann2001,sunday2021}. We note that
\begin{equation}
P_k \subseteq \operatorname{conv}\bm{\lambda}[A(0)].
\end{equation}
From our assumption~(\ref{eq:Case-II_condition_conv}),
\begin{equation}\label{eq:0_not_in_conv_lambda_0}
0 \notin \operatorname{conv}\bm{\lambda}[A(0)]. 
\end{equation}
Therefore,
\begin{equation} \label{eq:0_notin_Pk}
0 \notin P_k. 
\end{equation}
It is known that a point is not in a polygon if and only if the winding number of the polygon around that point is zero~\cite{hormann2001,sunday2021}. Therefore,
\begin{equation}
    \mu_k = 0. 
\end{equation}
This completes the proof of Eq.~(\ref{eq:claim_mu_k}) and thus Eq.~(\ref{eq:S_prime_property_5}). 

\end{itemize}
\end{proof} 

\section*{Declaration of competing interest}
We declaim that there is no competing interest.

\section*{Acknowledgements}
This work is funded by a Simons Investigator in Physics
grant from the Simons Foundation (Grant No. 827065)

%\appendix

%\section{Sample Appendix Section}

 \bibliographystyle{elsarticle-num} 
 \bibliography{main}





% \end{thebibliography}
\end{document}